\documentclass[11pt]{article}
\usepackage[margin=1in]{geometry}

\usepackage{amsmath} % assumes amsmath package installed
\usepackage{amssymb}  % assumes amsmath package installed
\usepackage{amsthm}
\usepackage{apxproof}

\makeatletter
\renewcommand*\env@matrix[1][c]{\hskip -\arraycolsep
  \let\@ifnextchar\new@ifnextchar
  \array{*\c@MaxMatrixCols #1}}
\makeatother

\usepackage{scalefnt}
\usepackage[hidelinks,colorlinks=true,linkcolor=blue,citecolor=blue]{hyperref}
\usepackage{color}

\theoremstyle{plain}  % Bold name, italics font
\newtheorem{theorem}{Theorem}[section]
\newtheorem{lemma}[theorem]{Lemma}

\newtheorem{corollary}[theorem]{Corollary}
\newtheorem{definition}[theorem]{Definition}

\theoremstyle{definition}

\theoremstyle{remark} % italics name, roman font

\newcommand{\rp}{\mathbb{RP}}
\newcommand{\rr}{\mathbb{R}}

\newcommand{\s}{S_{n,d}}

\def\x{\mathbf{x}}
\def\y{\mathbf{y}}
\def\z{\mathbf{z}}
\def\s{\mathbf{s}}
\def\v{\mathbf{v}}
\def\u{\mathbf{u}}
\def\w{\mathbf{w}}

\def\b{\mathbf{b}}
\def\c{\mathbf{c}}
\def\d{\mathbf{d}}

\long\def\aaa#1{{#1}}

\title{\LARGE \bf Convex Ternary Quartics Are SOS-Convex}
 \author{Amir Ali Ahmadi, Grigoriy Blekherman, and Pablo A. Parrilo \thanks{Amir Ali Ahmadi is with the Department of Operations Research and Financial Engineering at Princeton University (Email: \texttt{aaa@princeton.edu}). Grigory
Blekherman is with the School of Mathematics at the
Georgia Institute of Technology (Email: \texttt{greg@math.gatech.edu}). 
 Pablo A. Parrilo is with the Laboratory for Information and
Decision Systems and the Department of Electrical Engineering and
Computer Science at Massachusetts Institute of Technology (Email: \texttt{parrilo@mit.edu}).
}
\thanks{This research was partially supported by NSF FRG DMS-0757207, NSF CCF-1565235, by the MURI Award of the AFOSR, and by the Sloan Fellowship. GB: NSF DMS-1352073, DMS-1901950.}
\thanks{Version of \today .} }
\begin{document}
\date{}
\maketitle

%\thispagestyle{empty}
%\pagestyle{empty}

%%%%%%%%%%%%%%%%%%%%%%%%%%%%%%%%%%%%%%%%%%%%%%%%%%%%%%%%%%%%%%%%%%%%%%%%%%%%%%%%
\begin{abstract}
We show that if a ternary quartic form is convex, then it must be
sos-convex; i.e, if the Hessian $H(\x)$ of a ternary quartic form
is positive semidefinite for all $\x$, then the biquadratic form
$\y^TH(\x)\y$ in the variables
$\x\mathrel{\mathop:}=(x_1,x_2,x_3)^T$ and
$\y\mathrel{\mathop:}=(y_1,y_2,y_3)^T$ must be a sum of squares.
This result is in a meaningful sense the ``convex analogue'' of
Hilbert's celebrated theorem on ternary quartics. 
%\aaa{We also present an explicit example of a biquadratic form $b(\x,\y)$ that is symmetric in $\x$ and $\y$,
%\aaa{nonnegative}, but not a sum of squares. This shows that exploiting the structure of the Hessian matrix is crucial in any possible proof of main result.}
We show that exploiting the structure of the Hessian matrix is crucial in any possible proof of this result by presenting an explicit example of a biquadratic form $b(\x,\y)$ that is symmetric in $\x$ and $\y$, \aaa{nonnegative}, but not a sum of squares.
\end{abstract}

%%%%%%%%%%%%%%%%%%%%%%%%%%%%%%%%%%%%%%%%%%%%%%%%%%%%%%%%%%%%%%%%%%%%%%%%%%%%%%%%
\section{Introduction}
\aaa{A form (i.e., homogeneous polynomial) $p:\mathbb{R}^n\rightarrow\mathbb{R}$ with real coefficients is said to be \emph{nonnegative} if $p(x)\geq 0$ for all $x\in\mathbb{R}^n$ and a \emph{sum of squares} (sos) if it can be written as $p(x)=\sum_{i=1}^m q_i^2(x)$ for some forms $q_1(x),\ldots, q_m(x)$.} In 1888, Hilbert~\cite{Hilbert_1888} showed that \aaa{a ternary quartic form (i.e., a form in 3 variables of degree 4) is nonnegative if and only if it is a sum of squares.}
%homogeneous polynomials in~3 variables and degree~4 (i.e., ternary quartic forms) are positive semidefinite if and only if they can be written as a sum of squares of polynomials, in fact as a sum of three squares of quadratic forms.
Out of the degrees and dimensions in which \aaa{nonnegative} forms can be written as a sum of squares, the case of ternary quartics is
the most astonishing. Several new proofs of this result as well as
modern expositions of Hilbert's original proof have appeared in
recent years; see~\cite{Choi_Lam_extremalPSDforms}, \cite[p.
89-93]{Squares_Book}, \cite{Swan_ternary_quartic},
\cite{NewApproach_Hilbert_Ternary_Quatrics},
\cite{Scheiderer_ternary_quartic}.

In this paper, we show that interestingly enough an analogous fact
holds for convexity and sos-convexity of ternary quartic forms. A
polynomial $p\mathrel{\mathop:}=p(\x)$ in the variables
$\x\mathrel{\mathop:}=(x_1,\ldots,x_n)^T$ is \emph{convex} if its
Hessian \aaa{matrix} $H_p\mathrel{\mathop:}=H_p(\x)$ is positive semidefinite \aaa{(i.e., has nonnegative eigenvalues)} for all $\x\in\mathbb{R}^n$. It is easy to see that this condition holds if and only if the scalar polynomial $\y^TH_p(\x)\y$ in the variables $\x$ and $\y\mathrel{\mathop:}=(y_1,\ldots,y_n)^T$ is \aaa{nonnegative}. A polynomial $p$ is said to be \emph{sos-convex} if the polynomial $\y^TH_p(\x)\y$ is a sum of squares. Clearly,
sos-convexity is a sufficient condition for convexity of
polynomials.

The term ``sos-convexity'' was introduced by Helton and Nie
in~\cite{Helton_Nie_SDP_repres_2} in relation to the study of
semidefinite representability of convex sets. An alternative
definition of sos-convexity that is commonly used (e.g.,
in~\cite{Helton_Nie_SDP_repres_2}) and is equivalent to the
definition we gave above is given by the requirement that the
Hessian matrix $H_p$ can be factored as $H_p(\x)=M^T(\x)M(\x)$ for
a possibly nonsquare polynomial matrix $M(\x)$. There are also
other natural sos relaxations for convexity based on the usual
definition of convexity (via Jensen's inequality) or its first
order characterization. However, these sos relaxations are shown
in~\cite{AAA_PP_CDC10_algeb_convex},~\cite{AAA_PP_table_sos-convexity}
to also be equivalent to sos-convexity. All these equivalence
results confirm that sos-convexity is indeed \emph{the} rightful
analogue of sum of squares when the notion of convexity instead of
\aaa{nonnegativity} of polynomials is of interest.

From a computational viewpoint, the significance of sos-convexity
stems from the fact that it can be checked efficiently by solving
a single semidefinite program. This is in contrast to deciding
convexity which has \aaa{been} shown to be strongly NP-hard
already \aaa{for polynomials of} degree
four~\cite{NPhard_Convexity}. Motivated in part by its
connection to semidefinite programming, sos-convexity has \aaa{has found a range of applications, e.g., in the study of polynomial norms~\cite{polynomial_norms}, shape-constrained regression~\cite{MihaelaGeorgina_OR},~\cite{convex_fitting}, polynomial optimization~\cite{Lasserre_Jensen_inequality}, difference of convex optimization~\cite{AAA_GH_DC_Programming}, robust multi-objective optimization~\cite{multiobjective_sos-convex}, and dynamics and control~\cite{Chesi_Hung_journal},~\cite{AAA_Raph_sos-convex_cdc}. There has also been much interest in the role of
convexity in semialgebraic \aaa{optimization}~\cite{Lasserre_set_convexity},~\cite{Lasserre_Jensen_inequality},~\cite{Monique_Etienne_Convex},~\cite{Blekherman_convex_not_sos},~\cite{Helton_Nie_SDP_repres_2},~\cite{Bachir_SIAGA},~\cite{saunderson2023convex}, and an understanding of the relationship between convexity and
sos-convexity is of direct relevance to this line of research. For example, it is known that the semidefinite relaxation arising from the first level of the so-called sum of squares hierarchy is exact for polynomial optimization problems whose objective and constraints are given by sos-convex polynomials~\cite{Lasserre_Jensen_inequality}.}

% \aaa{appeared} in a number of theoretical and practical
% settings~\cite{Lasserre_Jensen_inequality},~\cite{Lasserre_Convex_Positive},~\cite{Helton_Nie_SDP_repres_2},~\cite{convex_fitting},~\cite{Chesi_Hung_journal}. 
% In particular, there has been much \aaa{interest} in the role of
% convexity in semialgebraic \aaa{optimization}~\cite{Lasserre_Jensen_inequality},~\cite{Monique_Etienne_Convex},~\cite{Blekherman_convex_not_sos},~\cite{Lasserre_set_convexity},
% and an understanding of the relation between convexity and
% sos-convexity is of major relevance to this area of research. \aaa{For example, it is known that the semidefinite program arising from the first level of the so-called sum of squares hierarchy is exact for polynomial optimization problems whose objective and constraints are given by sos-convex polynomials~\cite{Lasserre_Jensen_inequality}.}

In~\cite{AAA_PP_not_sos_convex_journal}, the first and third
authors gave the first example of a convex polynomial that is not
sos-convex. In a \aaa{subsequent}
paper~\cite{AAA_PP_table_sos-convexity}, they gave a full
characterization of the degrees and dimensions in which the set of
convex and sos-convex polynomials coincide. Such a
characterization is also given
in~\cite{AAA_PP_table_sos-convexity} for homogeneous polynomials,
except for the case of ternary quartics. The main contribution of
the current paper
(Theorem~\ref{thm:convex=sos.covex.ternary.quartics} in
Section~\ref{sec:proof.of.convexity=sos-convexity}) is to settle
this remaining case by showing that all convex ternary quartic
forms are sos-convex. The intriguing overall outcome of this
research is that convex polynomials (resp. forms) are sos-convex
exactly in degrees and dimensions where \aaa{nonnegative} polynomials (resp. forms) are sums of squares, as characterized by
Hilbert in~\cite{Hilbert_1888}. However, neither the results
in~\cite{AAA_PP_table_sos-convexity} nor the result of this paper
follow (as far as we know) from the characterization of Hilbert.
See the discussion in~\cite[Sec. 5]{AAA_PP_table_sos-convexity}.

Upon dehomogenization (see, e.g.,~\cite[Sec. 2]{Reznick}), the
result of Hilbert on ternary quartic forms is \aaa{equivalent} to the statement that all \aaa{nonnegative} bivariate quartic
polynomials are sos. The situation\aaa{, however, is} quite different for
convexity and sos-convexity. It turns out that the proof of the
fact that all convex bivariate quartic polynomials are sos-convex
follows from a theorem on factorization of positive semidefinite
bivariate and homogeneous polynomial matrices without the need to
exploit the additional structure of the Hessian matrix~\cite[Thm.
5.6]{AAA_PP_table_sos-convexity}; see also~\cite[Rmk.
5.1]{AAA_PP_table_sos-convexity}. \aaa{By contrast}, the result
for ternary quartic forms \aaa{crucially relies on linear relations that are imposed on the entries of a matrix that is a valid Hessian.}
%relies crucially and delicately on thelinear relations among the entries of the matrix that are imposed by the requirement of being a valid Hessian. 
In Section~\ref{sec:sym.biquad.counterexample}, we make this fact evident for the reader. We present an explicit example of a positive semidefinite polynomial matrix which has dimension, degree, and a special symmetry property in common with the Hessian of ternary quartics, but yet fails to have a sum of squares decomposition since it violates a few of the linear relations imposed by the Hessian structure.
%
%
%
%, of the same dimension and degree as the Hessian of ternary
%quartics and with a strong symmetry property imposed by the
%Hessian structure, which fails to have a sum of squares
%decomposition because it violates a few of the linear relations
%shared by all Hessian matrices.
%
%that satisfies strong symmetry properties imposed by the Hessian
%structure, but fails to have a sum of squares decomposition
%because it violates a few of the linear relations shared by all
%Hessian matrices.
%
%that even if a polynomial matrix of the appropriate dimension and
%degree satisfies strong symmetry properties imposed by the Hessian
%structure but not quite all of the linear relations of the Hessian
%matrix, then it is possible for it to be positive semidefinite
%without having a sum of squares factorization.
%
Following this result, we give the proof of our main theorem on
equivalence of convexity and sos-convexity of ternary quartics in
Section~\ref{sec:proof.of.convexity=sos-convexity}.

%In fact, we show in the next section (Theorem *) via an explicit
%counterexample that even if a polynomial matrix of the same
%dimension and degree satisfies most but not all of these linear
%relations, then it is possible for it to be positive semidefinite
%without having a sum of squares factorization.

\section{Symmetric biquadratic forms and \aaa{Hessian} biquadratic forms}\label{sec:sym.biquad.counterexample}

A \emph{biquadratic form} $b(\x,\y)$ is a form in two sets of
variables $\x=(x_1, \ldots, x_n)^T$ and $\y=(y_1, \ldots, y_m)^T$
that can be written as
\begin{equation}\nonumber
b(\x,\y)\mathrel{\mathop:}=\sum_{i\leq j, \, k\leq
l}\alpha_{ijkl}x_ix_jy_ky_l.
\end{equation}
\aaa{Equivalently, a biquadratic form is a quartic form that can
%Every biquadratic form is a quartic form, but the converse is of
%course not true. It is \aaa{straightforward} to see that any biquadratic form can
be written} as $$\y^TA(\x)\y,$$ where $A(\x)$ is a polynomial
matrix \aaa{with quadratic forms in $\x$ as its entries.}
%whose entries are quadratic forms in $\x$. 
The relation between
\aaa{nonnegative} and sum of squares biquadratic forms is a
well-studied subject. In particular, it \aaa{is known} that when
$nm\leq 6$, all \aaa{nonnegative} biquadratic forms are sos; see, e.g.,~\cite{CLRrealzeros} and references therein.

The question of checking \aaa{nonnegativity} of biquadratic
forms arises in the study of convexity of quartic forms. If
$H_p(\x)$ is the Hessian of a quartic form $p(\x)$, then the
entries of $H_p(\x)$ must be quadratic forms and hence $\y^TH_p(\x)\y$ is a biquadratic form. Convexity (resp.
sos-convexity) of $p$ is equivalent to this biquadratic form being
\aaa{nonnegative} (resp. sos).

In particular, when $p(\x)=p(x_1,x_2,x_3)$ is a
quartic form in three variables, the Hessian $H_p(\x)$ is a $3
\times 3$ matrix, and we are in the situation where
$\y^TH_p(\x)\y$ is a biquadratic form in two sets of three
variables $\x=(x_1,x_2,x_3)$ and
$\y=(y_1,y_2,y_3)$. It is well known that there
exist biquadratic forms in two sets of three
variables---henceforth referred to as \emph{ternary biquadratic
forms}---that are \aaa{nonnegative} but not sos. This fact was originally proven
through a nonconstructive argument by
Terpstra in~\cite{Biquadratic_Terpstra} and later independently by
Choi~\cite{Choi_Biquadratic} via an explicit example.
In~\cite{Choi_Biquadratic}, Choi showed that the biquadratic form
$\y^TC(\x)\y$ with
\begin{equation}\label{eq:choi.matrix}
C(\x)=\begin{bmatrix} x_1^2+2x_2^2&-x_1x_2&-x_1x_3 \\ \\
-x_1x_2&x_2^2+2x_3^2&-x_2x_3 \\ \\
-x_1x_3&-x_2x_3&x_3^2+2x_1^2
\end{bmatrix}
\end{equation}
is \aaa{nonnegative} but not sos. However, the matrix $C(\x)$ above is
\emph{not} a valid Hessian, i.e., it cannot be the matrix of the
second derivatives of any polynomial. If this was the case, the
third partial derivatives would commute. \aaa{But for this matrix, we
have, e.g.,}
$$\frac{\partial C_{1,1}(\x)}{\partial x_3}=0\neq-x_3=\frac{\partial C_{1,3}(\x)}{\partial
x_1}.$$

Our main result
in this paper
(Theorem~\ref{thm:convex=sos.covex.ternary.quartics}) can be
equivalently phrased as the statement that all \aaa{nonnegative ternary biquadratic forms that arise from valid Hessians are sos. It turns out that biquadratic forms that arise from valid Hessians have a special symmetry property. To facilitate discussion, let us define three families of biquadratic forms.
\begin{definition}\label{def:3.types.of.biquadratics}\ \begin{itemize}
    \item An \emph{$n$-ary biquadratic form} is a biquadratic form $b(\x,\y)$ in two sets of
$n$ scalar variables $\x=(x_1, \ldots, x_n)^T$ and $\y=(y_1, \ldots, y_n)^T$.
    \item A \emph{symmetric biquadratic form} is an ($n$-ary) biquadratic form that satisfies $b(\y,\x)=b(\x,\y)$.
    \item A \emph{Hessian biquadratic form} is a biquadratic form
$b(\x,\y)=\y^TH_p(\x)\y$ where $H_p(\x)$ is a valid Hessian, i.e., it is the matrix of second derivatives of some quartic form $p(\x)$.
\end{itemize}
\end{definition}

A simple counting argument shows that the dimension of the vector space of $n$-ary (resp. symmetric) biquadratic forms is ${n+1 \choose 2}^2$ (resp. $\frac{1}{2}{n+1 \choose 2}^2+\frac{1}{2}{n+1 \choose 2}$). 
%$\frac{n^2(n+1)^2}{4}$ (resp. $\frac{n^2(n+1)^2}{8}+\frac{n(n+1)}{4}$). 
We claim that the dimension of the vector space of Hessian biquadratic forms is the same as that of quartic forms in $n$ variables, i.e., $n+3 \choose 4$. This is because there is a linear bijection between these two vector spaces: from any quartic form $p$, we can obtain a valid Hessian matrix $H_p$ by differentiation; conversely, from any valid Hessian matrix $H_p$, we can produce the originating quartic form as\footnote{The identity $p(\x)=\frac{1}{d(d-1)} \x^TH_p(\x)\x$ holds for any form $p$ of degree $d$ and can be derived from Euler's identity for homogeneous functions. This identity also shows that convex forms are nonnegative, and that sos-convex forms are sos.}
\begin{equation}\label{eq:Euler.Hessian}
p(\x)=\frac{1}{12}\x^TH_p(\x)\x.    
\end{equation}

\begin{lemma}\label{lem:Hessian.biqauad.sym}
    Hessian biquadratic forms are symmetric.
\end{lemma}

\begin{proof}
    Consider a Hessian biquadratic form $\y^TH_p(\x)y$, where $H_p(\x)$ is the Hessian matrix of a quartic form $p(\x)$. Observe that if the quadratic form in the $ij$-th entry of $H_p(\x)$ is denoted by $\x^TS^{ij}\x,$ then the $kl$-the entry $S^{ij}_{kl}$ of the symmetric matrix $S^{ij}$ is given by $$S^{ij}_{kl}=\frac{1}{2}\partial_{x_i} \partial_{x_j} \partial_{x_k} \partial_{x_l} p(\x).$$ 
Since partial derivatives commute, we have $S^{ij}_{kl}=S^{kl}_{ij}$, and therefore $$\y^TH_p(\x)\y=\sum_{ij} y_iy_j \x^T S^{ij}\x=\sum_{ij} y_iy_j\sum_{kl}x_kx_l S^{ij}_{kl}=\sum_{ij} x_ix_j\sum_{kl}y_ky_l S^{ij}_{kl}=  
\sum_{ij} x_ix_j \y^T S^{ij}\y=\x^TH_p(\y)\x.$$

\end{proof}

}

\aaa{The symmetry of a biquadratic form in $\x$ and $\y$ is a rather strong
condition that is not satisfied e.g. by the Choi biquadratic form
$\y^TC(\x)\y$ in (\ref{eq:choi.matrix}) (since, in particular, there is a $2y_1^2x_2^2$ term but no term of the type $y_2^2x_1^2$). When $n=3$, the vector spaces
of $n$-ary biquadratic forms, symmetric biquadratic forms, and Hessian biquadratic forms respectively have dimensions
$36$, $21$, and $15$.} Since the symmetry requirement drops the
dimension of the space of \aaa{ternary} biquadratic forms significantly, and
since sos polynomials are known to generally cover much larger
volume in the set of \aaa{nonnegative} polynomials in presence of symmetries
(see, e.g.,~\cite{Symmetric_quartics_sos}), one may initially
suspect (as we did) that the equivalence between \aaa{nonnegative} and sos
ternary Hessian biquadratic forms is an artifact the symmetry
property. Our next theorem shows that interestingly enough this is
not the case. This makes the result for \aaa{Hessian} biquadratic forms
even more striking.

\begin{theorem}\label{thm:sym.biquad.psd.not.sos}
%There exists a ternary symmetric biquadratic form that is positive semidefinite but not a sum of squares.
\aaa{There exist ternary symmetric biquadratic forms that are nonnegative but not a sum of squares. In particular, the following biquadratic form has the desired properties:
\begin{equation}\label{eq:sym.biquad.psd.not.sos}
\begin{array}{rlll}
b(x_1,x_2,x_3,y_1,y_2,y_3)&=& 12(x_1^2y_1^2+x_2^2y_2^2+x_3^2y_3^2)\\ \  &\ &\ \\
\  &\
&+31x_1x_2y_1y_2-10x_1x_3y_1y_3-5x_2x_3y_2y_3\\ \  &\ &\ \\
\  &\
&
+12(x_2^2y_1^2+y_2^2x_1^2)+6(x_3^2y_1^2+y_3^2x_1^2)+12(x_2^2y_3^2+y_2^2x_3^2)

 \\ \  &\ &\ \\
\  &\
&

+4(x_1x_2y_1^2+y_1y_2x_1^2)+9(x_1x_3y_1^2+y_1y_3x_1^2)-10(x_2x_3y_1^2+y_2y_3x_1^2)

 \\ \  &\ &\ \\
\  &\
&+13(x_1x_3y_2^2+y_1y_3x_2^2)+13(x_2x_3y_2^2+y_2y_3x_2^2)+23(x_1x_2y_2^2+y_1y_2x_2^2)
 \\ \  &\ &\ \\
\  &\
&+5(x_1x_2y_3^2+y_1y_2x_3^2)+3(x_1x_3y_3^2+y_1y_3x_3^2)+7(x_2x_3y_3^2+y_2y_3x_3^2)
 \\ \  &\ &\ \\
\  &\ & +5(x_1x_2y_2y_3+y_1y_2x_2x_3)-11(x_1x_3y_2y_3+y_1y_3x_2x_3)+3(x_1x_3y_1y_2+y_1y_3x_1x_2).
\end{array}
\end{equation}
}
\end{theorem}

\aaa{The proof of this theorem appears in the appendix.}

\begin{toappendix}
In this appendix, we present the proof of Theorem~\ref{thm:sym.biquad.psd.not.sos}.
\begin{proof}
\aaa{For the convenience of the reader, let us recall the biquadratic form in~(\ref{eq:sym.biquad.psd.not.sos}):
\begin{equation}\nonumber%\label{eq:sym.biquad.psd.not.sos.appendix}
\begin{array}{rlll}
%%%Very Old version (don't know if it is correct)
%b(x_1,x_2,x_3;y_1,y_2,y_3)&=&
%4y_1y_2x_1^2+5x_1x_2y_3^2-10x_1x_3y_1y_3+5x_3^2y_1y_2+9x_1x_3y_1^2
%\\ \  &\ &\ \\
%\  &\ &
%+13x_1x_3y_2^2+3x_1x_3y_3^2+9y_1y_3x_1^2-10y_2y_3x_1^2-11x_1x_3y_2y_3+23x_2^2y_1y_2
%\\ \  &\ &\ \\
%\  &\ &
%+5y_2x_1x_2y_3-10x_2x_3y_1^2+13x_2^2y_1y_3+5x_2x_3y_1y_2-11x_2x_3y_1y_3
%\\ \  &\ &\ \\
%\  &\ &
%+13x_2^2y_2y_3+6x_3^2y_1^2+3x_1x_3y_1y_2-5x_2x_3y_2y_3+13x_2x_3y_2^2+12x_2^2y_2^2\\ \  &\ &\ \\
%\  &\ &
%+12x_2^2y_1^2+12x_3^2y_2^2+12x_3^2y_3^2+7x_2x_3y_3^2+3x_3^2y_1y_3+7x_3^2y_2y_3
%\\ \  &\ &\ \\
%\  &\ &
%+12x_2^2y_3^2+12y_1^2x_1^2+12y_2^2x_1^2+6y_3^2x_1^2+31y_1y_2x_1x_2
%\\ \  &\ &\ \\
%\  &\ &+3y_1y_3x_1 x_2+4y_1^2x_1x_2+23y_2^2x_1x_2.

%%%%%%%%%%%%%%%%%%%%%%%%%%%%%%%
%%%%%%%%%%%%%%%%%%%%%%%%%%%%%%%
%%Below is what was here on January 4, 2024 before I rearranged it.
% b(x_1,x_2,x_3,y_1,y_2,y_3)&=&4y_1y_2x_1^2+4x_1x_2y_1^2+9y_1y_3x_1^2+9x_1x_3y_1^2-10y_2y_3x_1^2-10x_2x_3y_1^2 \\ \  &\ &\ \\
% \  &\
% &+12y_1^2x_1^2+12y_2^2x_1^2+12x_2^2y_1^2+6y_3^2x_1^2+6x_3^2y_1^2+23x_2^2y_1y_2
%  \\ \  &\ &\ \\
% \  &\
% &+23y_2^2x_1x_2+13x_2^2y_1y_3+13x_1x_3y_2^2+13y_2y_3x_2^2+13x_2x_3y_2^2
%  \\ \  &\ &\ \\
% \  &\
% &+12x_2^2y_2^2+12x_2^2y_3^2+12y_2^2x_3^2+5x_3^2y_1y_2+5y_3^2x_1x_2+12x_3^2y_3^2
%  \\ \  &\ &\ \\
% \  &\ & +3x_3^2y_1y_3+3y_3^2x_1x_3+7x_3^2y_2y_3+7y_3^2x_2x_3 \\  \  &\ &\ \\
% \  &\
% &+31y_1y_2x_1x_2-10x_1x_3y_1y_3-11x_1x_3y_2y_3-11y_1y_3x_2x_3
%  \\ \  &\ &\ \\
% \  &\
% &+5x_1x_2y_2y_3+5y_1y_2x_2x_3+3x_1x_3y_1y_2+3y_1y_3x_1x_2-5x_2x_3y_2y_3.

b(x_1,x_2,x_3,y_1,y_2,y_3)&=& 12(x_1^2y_1^2+x_2^2y_2^2+x_3^2y_3^2)\\ \  &\ &\ \\
\  &\
&+31x_1x_2y_1y_2-10x_1x_3y_1y_3-5x_2x_3y_2y_3\\ \  &\ &\ \\
\  &\
&
+12(x_2^2y_1^2+y_2^2x_1^2)+6(x_3^2y_1^2+y_3^2x_1^2)+12(x_2^2y_3^2+y_2^2x_3^2)

 \\ \  &\ &\ \\
\  &\
&

+4(x_1x_2y_1^2+y_1y_2x_1^2)+9(x_1x_3y_1^2+y_1y_3x_1^2)-10(x_2x_3y_1^2+y_2y_3x_1^2)

 \\ \  &\ &\ \\
\  &\
&+13(x_1x_3y_2^2+y_1y_3x_2^2)+13(x_2x_3y_2^2+y_2y_3x_2^2)+23(x_1x_2y_2^2+y_1y_2x_2^2)
 \\ \  &\ &\ \\
\  &\
&+5(x_1x_2y_3^2+y_1y_2x_3^2)+3(x_1x_3y_3^2+y_1y_3x_3^2)+7(x_2x_3y_3^2+y_2y_3x_3^2)
 \\ \  &\ &\ \\
\  &\ & +5(x_1x_2y_2y_3+y_1y_2x_2x_3)-11(x_1x_3y_2y_3+y_1y_3x_2x_3)+3(x_1x_3y_1y_2+y_1y_3x_1x_2).
\end{array}
\end{equation}
}

% \aaa{[*Below is the original ordering. Let's make sure that it matches what I have above. Then we can rease it.*]}
% \begin{equation}\nonumber
% \begin{array}{rlll}
% %%%%%%%%%%%%%%%%%%%%%%%%%%%%%%%
% %%%%%%%%%%%%%%%%%%%%%%%%%%%%%%%
% %%Below is what was here on January 4, 2024 before I rearranged it.
% b(x_1,x_2,x_3,y_1,y_2,y_3)&=&4y_1y_2x_1^2+4x_1x_2y_1^2+9y_1y_3x_1^2+9x_1x_3y_1^2-10y_2y_3x_1^2-10x_2x_3y_1^2 \\ \  &\ &\ \\
% \  &\
% &+12y_1^2x_1^2+12y_2^2x_1^2+12x_2^2y_1^2+6y_3^2x_1^2+6x_3^2y_1^2+23x_2^2y_1y_2
%  \\ \  &\ &\ \\
% \  &\
% &+23y_2^2x_1x_2+13x_2^2y_1y_3+13x_1x_3y_2^2+13y_2y_3x_2^2+13x_2x_3y_2^2
%  \\ \  &\ &\ \\
% \  &\
% &+12x_2^2y_2^2+12x_2^2y_3^2+12y_2^2x_3^2+5x_3^2y_1y_2+5y_3^2x_1x_2+12x_3^2y_3^2
%  \\ \  &\ &\ \\
% \  &\ & +3x_3^2y_1y_3+3y_3^2x_1x_3+7x_3^2y_2y_3+7y_3^2x_2x_3 \\  \  &\ &\ \\
% \  &\
% &+31y_1y_2x_1x_2-10x_1x_3y_1y_3-11x_1x_3y_2y_3-11y_1y_3x_2x_3
%  \\ \  &\ &\ \\
% \  &\
% &+5x_1x_2y_2y_3+5y_1y_2x_2x_3+3x_1x_3y_1y_2+3y_1y_3x_1x_2-5x_2x_3y_2y_3.
% \end{array}
% \end{equation}

The fact that $b(\x,\y)=b(\y,\x)$ can readily be seen from the
order in which we have written the monomials. To prove that
$b(\x,\y)$ is \aaa{nonnegative}, we show that
\begin{equation}\label{eq:b*(x1^2+x2^2)}
b(\x,\y)(x_1^2+x_2^2)
\end{equation}
is sos. This, together with nonnegativity of $(x_1^2+x_2^2)$ and
continuity of $b(\x,\y)$, implies that $b(\x,\y)$ is \aaa{nonnegative}. A
rational sum of squares certificate for (\ref{eq:b*(x1^2+x2^2)}),
which we have obtained from the software package
SOSTOOLS~\cite{sostools}, is as follows:
$$b(\x,\y)(x_1^2+x_2^2)=\frac{1}{384}\z^TQ\z,$$ where $\z$ is the
vector of monomials
\begin{equation}\nonumber
\begin{array}{lll}
\z&=&[   x_2x_3y_3,
   x_2x_3y_2,
   x_2x_3y_1,
    x_2^2y_3,
    x_2^2y_2,
    x_2^2y_1,
   x_1x_3y_3,\\
   \ &\ &\
   x_1x_3y_2,
   x_1x_3y_1,
   x_1x_2y_3,
   x_1x_2y_2,
   x_1x_2y_1,
    x_1^2y_3,
    x_1^2y_2,
    x_1^2y_1]^T,
\end{array}
\end{equation}
and $Q$ is the positive definite matrix given by
\setcounter{MaxMatrixCols}{20} \scalefont{.55}
\begin{equation}\nonumber
Q=\begin{pmatrix}4608 & 1344 & 576 & 1344 & -504 & -264 & 0 & -900 & -264 & 1392 & -612 & -303 & 972 & -612 & 576 \\ \\
1344 & 4608 & 960 & -456 & 2496 & 2340 & 900 & 0 & 864 & 264 &
3072 & 1572 & 396 & 672 & 240 \\ \\
576 & 960 & 2304 & -1848 & -1380 & -1920 & 264 & -864 & 0 & -483 &
-1440 & 1200 & -888 & 240 & -1164 \\ \\
1344 & -456 & -1848 & 4608 & 2496 & 2496 & -816 & -1548 & -1047 &
960 & 840 & 24 & 1896 & -1944 & 1452 \\ \\
-504 & 2496 & -1380 & 2496 & 4608 & 4416 & -216 & -576 & 312 & 120
& 4416 & 1356 & 1380 & -284 & 1620 \\ \\
-264 & 2340 & -1920 & 2496 & 4416 & 4608 & -87 & 132 & 528 & 552 &
4596 & 768 & 1512 & -24 & 1892 \\ \\
0 & 900 & 264 & -816 & -216 & -87 & 4608 & 1344 & 576 & 372 &
-2400 & -1980 & 576 & -1572 & -2400 \\ \\
-900 & 0 & -864 & -1548 & -576 & 132 & 1344 & 4608 & 960 & 1656 &
1824 & -828 & -540 & 2496 & -84 \\ \\
-264 & 864 & 0 & -1047 & 312 & 528 & 576 & 960 & 2304 & 180 & 1308
& -756 & 480 & 660 & 1728 \\ \\
1392 & 264 & -483 & 960 & 120 & 552 & 372 & 1656 & 180 & 3120 &
1140 & 1260 & 960 & 876 & 1152 \\ \\
-612 & 3072 & -1440 & 840 & 4416 & 4596 & -2400 & 1824 & 1308 &
1140 & 9784 & 3588 & 84 & 4416 & 3660 \\ \\
-303 & 1572 & 1200 & 24 & 1356 & 768 & -1980 & -828 & -756 & 1260
& 3588 & 5432 & -576 & 2292 & 768 \\ \\
972 & 396 & -888 & 1896 & 1380 & 1512 & 576 & -540 & 480 & 960 &
84 & -576 & 2304 & -1920 & 1728 \\ \\
-612 & 672 & 240 & -1944 & -284 & -24 & -1572 & 2496 & 660 & 876 &
4416 & 2292 & -1920 & 4608 & 768 \\ \\ 576 & 240 & -1164 & 1452 &
1620 & 1892 & -2400 & -84 & 1728 & 1152 & 3660 & 768 & 1728 & 768
& 4608
\end{pmatrix}.
\end{equation}
\normalsize Let us now prove that $b$ is not sos. If we denote the
cone of sos ternary biquadratic forms by $\Sigma_{B,3}$ and its
dual cone by $\Sigma_{B_,3}^*$, our proof will simply proceed by
presenting a dual functional $\xi\in\Sigma_{B_,3}^*$ that takes a
negative value on the polynomial $b$. Let us fix the following
ordering for monomials of ternary biquadratic forms:
\begin{equation}\label{eq:monomial.ordering}
\begin{array}{ll}
\ &\{
         x_3^2y_3^2,
         x_3^2y_2y_3,
         x_3^2y_2^2,
         x_3^2y_1y_3,
         x_3^2y_1y_2,
         x_3^2y_1^2,
         x_2x_3y_3^2,
         x_2x_3y_2y_3,
         x_2x_3y_2^2,
         x_2x_3y_1y_3,
         x_2x_3y_1y_2,
         x_2x_3y_1^2,\\
         \ &\\
         \ &
         x_2^2y_3^2,
         x_2^2y_2y_3,
         x_2^2y_2^2,
         x_2^2y_1y_3,
         x_2^2y_1y_2,
         x_2^2y_1^2,
         x_1x_3y_3^2,
         x_1x_3y_2y_3,
         x_1x_3y_2^2,
         x_1x_3y_1y_3,
         x_1x_3y_1y_2,
         x_1x_3y_1^2,\\
         \  &\\
         \ &
         x_1x_2y_3^2,
         x_1x_2y_2y_3,
         x_1x_2y_2^2,
         x_1x_2y_1y_3,
         x_1x_2y_1y_2,
         x_1x_2y_1^2,
         x_1^2y_3^2,
         x_1^2y_2y_3,
         x_1^2y_2^2,
         x_1^2y_1y_3,
         x_1^2y_1y_2,
         x_1^2y_1^2\}.
\end{array}
\end{equation}
With this ordering, the vector of coefficients $\vec{\b}$ of the
biquadratic form $b$ in (\ref{eq:sym.biquad.psd.not.sos}) is given
by
\begin{equation}\nonumber
\begin{array}{ll}
\vec{\b}=&[12,
    7,
    12,
    3,
    5,
    6,
    7,
    -5,
    13,
    -11,
    5,
    -10,
    12,
    13,
    12,
    13,
    23,\\
\ &    12,
    3,
    -11,
    13,
    -10,
    3,
    9,
    5,
    5,
    23,
    3,
    31,
    4,
    6,
    -10,
    12,
    9,
    4,
    12]^T.
\end{array}
\end{equation}
Using the same ordering, we can represent our dual functional
$\xi$ with the vector
\begin{equation}\nonumber
\begin{array}{ll}
\c=&[ 37,
   -18,
    18,
   -23,
    -1,
    66,
   -18,
    12,
   -15,
     1,
    -1,
    35,
    18,
   -15,
    96,
    -5,
   -37,\\
\ & 64,
   -23,
     1,
    -5,
    34,
    -7,
   -48,
    -1,
    -1,
   -37,
    -7,
   -15,
     0,
    66,
    35,
    64,
   -48,
     0,
    61]^T.
\end{array}
\end{equation}
We have
\begin{equation}\nonumber
\langle\xi,b\rangle=\c^{T}\vec{\b}=-37<0.
\end{equation}
On the other hand, we claim that $\xi\in\Sigma_{B,3}^*$; i.e., for
any form $w\in\Sigma_{B,3}$, we should have
\begin{equation}\label{eq:c.w>=0}
\langle\xi,w\rangle=\c^{T}\vec{\w}\geq0,
\end{equation}
where $\vec{\w}$ here denotes the coefficients of $w$ listed
according to the ordering in (\ref{eq:monomial.ordering}). Indeed,
if $w$ is sos, then it can be written in the form
\[
w(x)=\tilde{\z}^{T}\tilde{Q}\tilde{\z}= \mathrm{Tr} \ \tilde{Q}
\cdot \tilde{\z}\tilde{\z}^{T},
\]
for some symmetric positive semidefinite matrix $\tilde{Q}$, and a
vector of monomials
\[
\tilde{\z}=[   x_1y_1,
    x_1y_2,
    x_1y_3,
    x_2y_1,
    x_2y_2,
    x_2y_3,
    x_3y_1,
    x_3y_2,
    x_3y_3]^T.
\]
It is not difficult to see that
\begin{equation}\label{eq:c.vec(w)=traceQzzz'}
\c^{T}\vec{\w}= \mathrm{Tr}  \, \tilde{Q} \cdot
(\tilde{\z}\tilde{\z}^{T}) \vert_{\c},
\end{equation}
where by $(\tilde{\z}\tilde{\z}^{T})\vert_{\c}$ we mean a matrix
where each monomial in $\tilde{\z}\tilde{\z}^{T}$ is replaced with
the corresponding element of the vector $\c$. This yields the
matrix
\[(\tilde{\z}\tilde{\z}^{T})\vert_{\c}=
\begin{pmatrix}[r]
 61 & 0 & -48 & 0 & -15 & -7 & -48 & -7 & 34 \\
0 & 64 & 35 & -15 & -37 & -1 & -7 & -5 & 1 \\
-48 & 35 & 66 & -7 & -1 & -1 & 34 & 1 & -23 \\
0 & -15 & -7 & 64 & -37 & -5 & 35 & -1 & 1 \\
-15 & -37 & -1 & -37 & 96 & -15 & -1 & -15 & 12 \\
-7 & -1 & -1 & -5 & -15 & 18 & 1 & 12 & -18 \\
-48 & -7 & 34 & 35 & -1 & 1 & 66 & -1 & -23 \\
-7 & -5 & 1 & -1 & -15 & 12 & -1 & 18 & -18 \\
34 & 1 & -23 & 1 & 12 & -18 & -23 & -18 & 37
\end{pmatrix},
\]
which can easily be checked to be positive definite. Therefore,
equation (\ref{eq:c.vec(w)=traceQzzz'}) along with the fact that
$Q$ is positive semidefinite implies that (\ref{eq:c.w>=0}) holds.
This completes the proof.\footnote{\aaa{For verification purposes, we
have made the content of this proof available in electronic form
at \texttt{http://aaa.princeton.edu/software}. In particular, whenever we state that a matrix is positive definite, this claim is certified by a rational $LDL^T$
factorization.}}
\end{proof}

\end{toappendix}

\section{Equivalence of convexity and sos-convexity for ternary quartics}\label{sec:proof.of.convexity=sos-convexity}
Let us denote the set of convex (resp. sos-convex) ternary quartic
forms by $C_{3,4}$ (resp. $\Sigma C_{3,4}$). These sets are both
closed convex cones and we have the obvious inclusion $\Sigma C_{3,4}\subseteq
C_{3,4}$. Our main result is to show the reverse inclusion.

\begin{theorem}\label{thm:convex=sos.covex.ternary.quartics}
$\Sigma C_{3,4}=C_{3,4}.$
\end{theorem}

The proof of this theorem is done in two steps. We first show that
it suffices to consider convex forms that have a
specific set of \aaa{zeroes} (Lemma~\ref{lem:reduction}). We then show
that all such convex forms are sos-convex (Theorem~\ref{THM Cone
Desc}). Throughout this section, we use the notation $H_p$ to
denote the Hessian matrix of a form $p$, and
$h_p(\x,\y)=\y^TH_p\y$ to denote the \textit{Hessian form} of $p$.
We recall that when $p$ is a quartic, $h_p$ is a biquadratic form
and satisfies the symmetry relation $h_p(\x,\y)=h_p(\y,\x)$; \aaa{see Lemma~\ref{lem:Hessian.biqauad.sym}}. Zeroes of $h_p$ are treated as points in $\mathbb{R}\mathbb{P}^2\times \mathbb{R}\mathbb{P}^2$; for evaluation, we may choose an affine representative, which will usually be taken to lie in \aaa{the bi-sphere} $\mathbb{S}^2\times \mathbb{S}^2$.

%$\mathbb{S}^{2}\times
%\mathbb{S}^{2}$ (or $\mathbb{R}\mathbb{P}^2\times \mathbb{R}\mathbb{P}^2$ FORMULATE BETTER.).

%Let $p$ be a form in $\h$. Let $H_p$ denote the Hessian matrix of
%$p$ and let $h_p(\x,\y)=y^TH_py$ denote the \textit{Hessian form}
%of $p$. We note that $h_p$ is a bi-homogeneous form, quadratic in
%$\y$ and of degree $2d-2$ in $\x$. Therefore, we treat zeroes of
%$h_p$ as points in $\rp^{n-1}\times \rp^{n-1}$. We also note that
%for ternary quartics $p$ we have $h_p(\x,\y)=h_p(\y,\x)$.

\subsection{Reduction}

%Let $C_{n,2d}$ be the cone of convex forms in $n$ variables of
%degree $2d$. Equivalently $C_{n,2d}$ consists of all forms, whose
%Hessian forms are non-negative. Let $A_{n,2d}$ denote the cone of
%sos-convex forms.

For a point $\u=(\v_1,\v_2) \in \mathbb{R}\mathbb{P}^2\times \mathbb{R}\mathbb{P}^2$, let $F_{\u}$ be the face of $C_{3,4}$ consisting of all
convex forms $f$ for which $h_f(\u)=0$. \aaa{Let} $\mathbf{e}_i$ denote the $i$-th standard basis vector, $\u_1=\left(\mathbf{e}_1,\mathbf{e}_2\right)^T,$ and
$\u_2=(\mathbf{e_3},\d)^T$ with \aaa{$\d=[a,b,1]^T$ for some scalars $a,b$}. Let $L_{a,b}$ be the linear subpsace of ternary quartics consisting of forms $g$ such that $H_g(\mathbf{e_1})\cdot \mathbf{e}_2=H_g(\mathbf{e_3})\cdot \d=0$.
Let $T_{a,b}$ be the face of $C_{3,4}$ consisting of convex ternary quartics $g$ such that $h_g(\u_1)=h_g(\u_2)=0$. Note that $T_{a,b} = F_{\u_1}\cap F_{\u_2}$. In Theorem \ref{THM Cone Desc}, we will derive an explicit description of the face $T_{a,b}$, and using it we will show that all forms in $T_{a,b}$ are sos-convex. Our main task for now is to show that if a convex and non-sos-convex ternary quartic exists, then there exists one in $T_{a,b}$. We start with some preparatory lemmas.

\begin{lemma}\label{lem:dimFu1}
The face $F_{\mathbf{u_1}}$ has dimension $10$. 
\end{lemma}
\begin{proof}
We first show that the dimension of $F_{\mathbf{u_1}}$ is at most 10. Let $g$ be a convex ternary quartic such that $h_g(\mathbf{e}_1,\mathbf{e_2})=0$. \aaa{Since the Hessian matrices $H_g(\mathbf{e}_1)$ and $H_g(\mathbf{e}_2)$ are positive semidefinite,} it follows that $H_g(\mathbf{e}_1)\cdot \mathbf{e}_2=H_g(\mathbf{e}_2)\cdot \mathbf{e}_1=0$. These two conditions imply that $g$ is missing monomials $x_1^3x_2,x_1^2x_2^2,x_1^2x_2x_3,x_2^3x_1,x_2^2x_1x_3$, which shows that $F_{\mathbf{u_1}}$ lies in a 10-dimensional subspace.

We now show that dimension of $F_{\mathbf{u_1}}$ is at least 10. Let $$f=(x_1+x_3)^4+(x_2+x_3)^4+(2x_1+x_3)^4+(2x_2+x_3)^4.$$ %We show that $f$ lies in the relative interior of $F_\mathbf{u_1}$. 
The Hessian form $h_f$ is equal to $$12((x_1+x_3)^2(y_1+y_3)^2+(x_2+x_3)^2(y_2+y_3)^2+(2x_1+x_3)^2(2y_1+y_3)^2+(2x_2+x_3)^2(2y_2+y_3)^2).$$ We see that $f\in F_{\mathbf{u_1}}$ and $h_f$ is strictly positive on $\mathbb{S}^2\times \mathbb{S}^2$ outside of $\mathbf{u_1}$. A straightforward computation shows that the Hessian of $h_f$ at $\mathbf{u_1}$ is positive definite on the tangent space to $\mathbb{S}^2\times\mathbb{S}^2$ at $\mathbf{u_1}$. Therefore, a perturbation argument shows that for any $g$ in the span of %10 monomials excluding 
\aaa{the 10 degree-4 monomials that are different from}
$x_1^3x_2,x_1^2x_2^2,x_1^2x_2x_3,x_2^3x_1,x_2^2x_1x_3,$ we have $f+\varepsilon g \in F_\mathbf{u_1}$ for all sufficiently small $\varepsilon$. It follows that $F_\mathbf{u_1}$ is 10-dimensional. 

\end{proof}

\begin{lemma}\label{lem:dimcount}
Whenever at least one of $a$ or $b$ is non-zero, the dimension of the vector space $L_{a,b}$ is $5$ and the dimension of $T_{a,b}$ is at most 5.
\end{lemma}

\begin{proof}

We begin by \aaa{establishing} a slightly more general version of the argument in Lemma \ref{lem:dimFu1}.
%We show that the dimension of the span of $T_{a,b}$ is at most $5$. 
For a vector $\v\in \mathbb{R}^3$, let $D_\v$ denote the directional derivative in direction $\v$: $D_\v(f)=\langle \v,\nabla{f}\rangle$. It will be convenient to consider differential operators associated to polynomials: for a polynomial $f,$ let $\partial f$ denote the differential operator obtained by replacing $x_i$ with $\frac{\partial}{\partial x_i}$.

Let $g$ be a convex ternary quartic such that $h_g(\u,\v)=0$, with non-zero $\u,\v$ in $\mathbb{R}^3$, $\u\neq\v$. \aaa{Since the Hessian matrices $H_g(\u)$ and $H_g(\v)$ are positive semidefinite,} it follows that $H_g(\u)\cdot \v=H_g(\v)\cdot \u=0$, i.e., $\u$ is in the kernel of the Hessian of $g$ evaluated at $\v$, and $\u$ is in the kernel of the Hessian of $g$ evaluated at $\u$.

 The condition $H_g(\u)\cdot \v=0$ implies that at the point $\u$, for any vector $\mathbf{w},$ the directional derivative $D_{\mathbf{w}}D_{\v}(g)$ is equal to $0$, i.e., we have $(D_{\mathbf{w}}D_{\v}(g))(\u)=0$. This is equivalent to the equation $D_{\u}D_{\u}D_{\mathbf{w}}D_{\v}(g)=D_{\mathbf{w}}D_{\u}D_{\u}D_{\v}(g)=0$ for any vector $\mathbf{w}\in \rr^3$.

We therefore see that the conditions on the Hessian $H_g$ of $g$ are equivalent to $g$ satisfying the following linear conditions: $D_{\u}D_{\u}D_{\v}(g)=0$ and $D_{\v}D_{\v}D_{\u}(g)=0$. Using differential operators, we can write the above conditions as $$\partial (\tilde{u}^2\tilde{v})[g]=\partial (\tilde{v}^2\tilde{u})[g]=0,$$ where $\tilde{u}=u_1x_1+u_2x_2+u_3x_3$ and $\tilde{v}=v_1x_1+v_2x_2+v_3x_3$ are the associated linear forms. 
We can rephrase this in the following way: for any ternary quartic $q$ which is a multiple of $\tilde{u}^2\tilde{v}$ or $\tilde{v}^2\tilde{u}$, we have $\partial (q) [f]=0$. In other words, the number of conditions imposed on $g$ is equal to the number of (linearly independent) quartics generated by $\tilde{u}^2\tilde{v}$ and $\tilde{v}^2\tilde{u}$. We see that there are $5$ conditions, since both $\tilde{u}^2\tilde{v}$ and $\tilde{v}^2\tilde{u}$ generate $\tilde{u}^2\tilde{v}^2$, and this is the only intersection between the ideals generated by $\tilde{u}^2\tilde{v}$ and $\tilde{v}^2\tilde{u}$ in degree $4$.

Consider zeroes at the two points $\mathbf{u_1}=(\mathbf{e}_1,\mathbf{e}_2)^T$ and $\mathbf{u_2}=(\mathbf{e_3},[a,b,1])^T$. From the first zero we get cubics $x_1^2x_2$ and $x_2^2x_1$, which generate the span of the following five degree $4$ monomials: $x_1^3x_2,x_1^2x_2^2,x_1^2x_2x_3,x_2^3x_1,x_2^2x_1x_3$. Note that only $x_1^2x_2x_3$ and $x_2^2x_1x_3$ are divisible by $x_3$.

From the second zero we get two distinct cubics: $x_3^2(ax_1+bx_2+x_3)$ and $(ax_1+bx_2+x_3)^2x_3$. Any quartic that these cubics generate has to be divisible by  $x_3(ax_1+bx_2+x_3)$. So the only possible intersection with the quartics from the first zero comes as linear combination $\alpha x_1^2x_2x_3+\beta x_2^2x_1x_3=x_1x_2x_3 (\alpha x_1+\beta x_2)$. But this linear combination is not divisible by $x_3(ax_1+bx_2+x_3)$. So we see that two zeroes impose 10 linearly independent conditions, and therefore the dimension of $L_{a,b}$ is $15-10=5$. Since $T_{a,b}$ is contained in $L_{a,b}$, the dimension of $T_{a,b}$ is at most $5$.
\end{proof}
%$a\neq 0$, $b \neq 0$ be two
%points in $\rp^2\times\rp^2$.

%Let $C_{n,2d}^*$ be the dual cone of $C_{n,2d}$:
%$$C_{n,2d}^*=\left\{\ell \in \h^* \st \ell(p) \geq 0 \quad \text{for all} \quad p \in C_{n,2d}\right\}.$$

\begin{lemma}[Reduction Lemma]\label{lem:reduction}
Suppose that $\Sigma C_{3,4}\neq C_{3,4}$. Then there exists a
form $p \in T_{a,b}$ with both $a$ and $b$ non-zero, such that $p \notin \Sigma C_{3,4}$.
\end{lemma}
\begin{proof}
%\aaa{[Changed some polynomials in this proof towards the end
%because I thought there were some minor typos. Please check.]} 
If
$\Sigma C_{3,4}\neq C_{3,4}$ then there exists a form $p$ on the
boundary of $C_{3,4}$ that is not in $\Sigma C_{3,4}$. The
boundary of $C_{3,4}$ consists of convex forms whose Hessian forms have a
nontrivial zero in $\rp^2 \times \rp^2$. Therefore, we may assume
that $h_p(\c,\d)=0$ for some point $(\c,\d) \in \rp^2 \times
\rp^2$.

We first observe that $\c\neq \d$. Suppose not and $h_p(\c,\c)=0$. \aaa{In view of (\ref{eq:Euler.Hessian}), this implies that $p(\c)=0$, and hence $p$ is a convex ternary form with a nontrivial zero.}
%By Euler's identity, it follows
%\footnote{One way to see why $h_p(\c,\c)=0$ implies
%$p(\c)=0$ is via the identity
%$p(\x)=\frac{1}{d(d-1)}\x^TH_p(\x)\x$ which holds for all forms of
%degree $d$ and is easily derived from Euler's homogeneous function
%theorem. The identity also demonstrates that convex forms are psd
%and that sos-convex forms are sos.} 
%that $p(\c)=0$ and $p$ is a convex form. 
By an observation of Reznick \cite[Prop. 4.1]{Blenders_Reznick},  it
follows that $p$ is a bivariate form defined on the orthogonal
complement of $\c$ in $\rr^3$. But then $p$ is
sos-convex~\cite[Thm. 5.4]{AAA_PP_table_sos-convexity}, which is a
contradiction. 

% \aaa{[An aside: The same argument implies that all
% ternary convex forms that have a zero are sos-convex and sos. The
% same is true for such convex forms in $4$ variables and degree
% $4$. This is interesting I think.]}

We can apply a nonsingular linear change of coordinates and move
the zero of $h_p$ to $\u_1$. The new form will still be convex but
not sos-convex. Therefore, we may assume that $h_p$ has a zero at
$\u_1$. Recall that $F_{\u_1}$ is the face of $C_{3,4}$ consisting of all
forms $f$ for which $h_f(\u_1)=0$. It follows that
there exists a form $\hat{p}$ in the relative interior of
$F_{\u_1}$ such that $\hat{p} \notin \Sigma C_{3,4}$. Note that
this implies that $h_{\hat{p}}(\x,\y)$ has a single zero at
$\u_1$, and moreover the Hessian of $h_{\hat{p}}$ is positive definite on the tangent space to $\mathbb{S}^2\times\mathbb{S}^2$ at $\mathbf{u_1}$.

%and $\mathbf{u_1}$ is a strict local minimum of $h_{\hat{p}}(\x,\y)$ in $\mathbb{S}^2\times \mathbb{S}^2$. \gb{It has the strict local minimum of any biquadratic form; equivalently the Hessian matrix of the biquadratic has rank 4, and it is PSD.} %\gb{we also want to say that this zero is ``simple"}.

Let $q=\frac{1}{12} x_3^4$, then we have $h_q(\x,\y)=x_3^2y_3^2$.
We observe that the Hessian matrix of $h_q$ is identically zero at
the point $\u_1$. Therefore for small enough $\varepsilon,$ we
have that $\hat{p}-\varepsilon q$ is still convex. Now let $\varepsilon>0$ be
such that $\bar{p}=\hat{p}-\varepsilon q$ is on the relative boundary
of $F_{\u_1}$. We note that $\bar{p}$ is convex but not
sos-convex. Since the Hessian matrix of $h_q$ is identically zero at
the point $\u_1$, it follows that the Hessian of $h_{\bar{p}}$ is positive definite on the tangent space to $\mathbb{S}^2\times\mathbb{S}^2$ at $\mathbf{u_1}$.
%that $\u_1$ remains a strict local minimum of $h_{\bar{p}}$, and 
Therefore
$h_{\bar{p}}$ must acquire an
additional zero at a point $\v=(\bar{\c},\bar{\d})$, $\bar{\c},\bar{\d}\in \mathbb{S}^2$. Since
$\bar{p}$ is not sos-convex, we see that $\bar{\c} \neq \bar{\d}$.
Furthermore, we must have $\bar{c}_3,\, \bar{d}_3 \neq 0$.
Otherwise, $x_3^2y_3^2$ is zero at the point $\bar{\c},\bar{\d}$
and thus $h_{\bar{p}}$ has the same value at $(\bar{\c},\bar{\d})$
as $h_{\hat{p}}$, while $h_{\hat{p}}$ had no zeroes outside of
$\u_1$. In other words, $\bar{\c}$ and $\bar{\d}$ are not in the
span of $[1,0,0]^T$, $[0,1,0]^T$.

For any $\v=(\c,\d)$ such that $\c$ and $\d$ are not in the
span of $[1,0,0]^T$, $[0,1,0]^T$, we claim that the face $F_{\c,\d}$ of $C_{3,4}$ consisting of convex ternary quartics with zeroes at $\mathbf{u_1}$ and $\v$ is at most $5$-dimensional. Apply an invertible linear change of coordinates that
fixes $[1,0,0]^T$, $[0,1,0]^T$ and maps $\bar{\c}$ to $[0,0,1]^T$. Since $\d$ is not in the span of $\mathbf{e_1}$ and $\mathbf{e_2}$ and $\mathbf{c}\neq \d$, the point $\d$
will be taken to (a non-zero multiple of) $[a,b,1]^T$ for some $a,b \in \mathbb R$ and $a$, $b$ not both zero. We apply Lemma \ref{lem:dimcount} to see that the dimension of $F_{\c,\d}$ is at most 5.

Note that there is an open ball $B$ of forms around $\hat{p}$ that are convex and not sos-convex. We can push any form in $B$ to the boundary of $F_{\mathbf{u_1}}$ by subtracting an appropriate multiple of $q$. It follows that we cover an open neighborhood $B'$ of $\bar{p}$ in the boundary of $F_{u_1}$ by convex ternary quartics that acquire a second zero at some (not necessarily same) point $\v=(\bar{\c},\bar{\d})$, with $\bar{\c}\neq\bar{\d}$ and $c_3$, $d_3$ not equal to $0$. 
%If $a=b=0$, then the resulting ternary quartic is SOS-convex, since it is convex and has a zero at $(0,0,1)$. %If $a$ and $b$ are both non-zero, then we have an explicit description of all convex ternary forms with these two zeroes below, and it follows that they ar all SOS-convex. If just one of $a$ or $b$ is zero, then we apply dimension counting argument, to show that forms with these zeroes form at most an 8-dimensional subset of $F_u$, which implies that they cannot cover all of the boundary.
%It remains to consider the case where exactly one of $a,b$ is non-zero. 
It remains to show that there exists a point in $B'$ for which the additional zero $\v=(\bar{\c},\bar{\d})$ is such that the four points $[1,0,0]^T$, $[0,1,0]^T$, $\c$ and $\d$ are in general linear position in $\mathbb{R}^3$. If this is the case, then an invertible linear change of coordinates can map $\c$ to $[0,0,1]^T$ and $\d$ to $[a,b,1]^T$ with both $a,b$ non-zero.

By Lemma \ref{lem:dimFu1}, the face $F_{\mathbf{u_1}}$ is 10-dimensional, and therefore its boundary is $9$-dimensional. We now derive a contradiction to all forms in $B'$ acquiring a zero $\v=(\c,\d)$, where $[1,0,0]^T$, $[0,1,0]^T$, $\c$ and $\d$ are not in general linear position via a dimension counting argument. As we saw above, the face $T_{\c,\d}$ is at most 5-dimensional. The pairs $(\c,\d) \in \mathbb{S}^2\times \mathbb{S}^2$ such that $[1,0,0]^T$, $[0,1,0]^T$, $\c$~and $\d$ are not in general linear position form a $3$-dimensional family. Therefore, all together such faces $T_{\c,\d}$ 
cover at most an $8$-dimensional subset of the boundary of $F_{\u_1}$, and hence they cannot cover all of $B'$.
\end{proof}

\subsection{Cone Description}

We now derive an explicit desciption of the face $T_{a,b}$ with both $a$ and $b$ non-zero. Let \begin{align*}q_1=x_1^4, \quad q_2=x_2^4, \quad
q_3=(x_1-ax_3)^4, \quad q_4=(x_2-bx_3)^4, \quad
q_5=x_3^2(bx_1-ax_2)^2.\end{align*} 
%\gb{Maybe we need to mention that the polynomials $q_i$ are linearly independent, and thus span the linear hull of $T_{a,b}$?E.g. We observe that $q_i$ are five linearly independent ternary quartics, such that $h_{q_i}$}
We have the following:
%description of $T_{a,b}$:

\begin{theorem}\label{THM Cone Desc}
The face $T_{a,b}$ consists of all forms
$\alpha_1q_1+\cdots+\alpha_5q_5$ such that
$\alpha_1,\ldots,\alpha_4 \geq 0$ and
$$-\frac{4a^2b^2}{\frac{b^4}{\alpha_1}+\frac{a^4}{\alpha_2}+\frac{b^4}{\alpha_3}+\frac{a^4}{\alpha_4}}\leq
\alpha_5 \leq 0.$$ Furthermore, all forms in $T_{a,b}$ are
sos-convex.
\end{theorem}

In view of Lemma~\ref{lem:reduction}, we note that a proof of
Theorem~\ref{THM Cone Desc} would complete the proof of
Theorem~\ref{thm:convex=sos.covex.ternary.quartics}. We start by
proving the latter claim of Theorem~\ref{THM Cone Desc}. Let
$S_{a,b}$ be the cone of all forms
$p=\alpha_1q_1+\cdots+\alpha_5q_5$ such that
$\alpha_1,\ldots,\alpha_4 \geq 0$ and \begin{equation}\label{EQN
Condition
Alpha}-\frac{4a^2b^2}{\frac{b^4}{\alpha_1}+\frac{a^4}{\alpha_2}+\frac{b^4}{\alpha_3}+\frac{a^4}{\alpha_4}}\leq
\alpha_5 \leq 0.\end{equation}

\begin{lemma}\label{lem:Sa,b.sos-convex}
Suppose that $p \in S_{a,b}$. Then $p$ is sos-convex.
\end{lemma}

\begin{proof}%\label{LEMMA Sos-Convex}
Let $p=\alpha_1q_1+\cdots+\alpha_5q_5 \in S_{a,b}$. First we
observe that if \aaa{for some $i\in\{1,\ldots,4\}$, $\alpha_i=0$,} then condition \eqref{EQN Condition Alpha} implies that $\alpha_5=0$
and then $p$ is a sum of fourth powers of linear forms and hence sos-convex. Therefore we may restrict ourselves
to the case of strictly positive $\alpha_1, \ldots, \alpha_4$. We
establish that $p$ is sos-convex by showing that $h_p(\x,\y)$ is a
sum of squares.

We will use an explicit set of squares in our decomposition. For
this we need a basis of the linear subspace of bilinear forms in
$\x$ and $\y$ with zeroes at $\u_1$ and $\u_2$. Let
\begin{align}\label{EQN Definition s_i}s_1=x_1y_1, \,\,
s_2=x_2y_2, \,\, s_3=(x_1-ax_3)(y_1-ay_3), \,\,
s_4=(x_2-bx_3)(y_2-by_3), \,\, s_5=x_3(by_1-ay_2).\end{align}

The forms $s_i$ were chosen so that for $i\in\{1,\ldots,4\}$ we have
$h_{q_i}(\x,\y)=s_i^2$. Let $\s=(s_1,s_2,s_3,s_4,s_5)^T$ and let
$c=\frac{a}{b}$. Consider the matrix:
\begin{equation}\label{EQN Matrix M}M=\left( \begin{array}{ccccc}
12\alpha_1+2\alpha_5c^{-2}&-2\alpha_5&-2\alpha_5c^{-2}&2\alpha_5&2\alpha_5c\\
-2\alpha_5&12\alpha_2+2\alpha_5c^{2}&2\alpha_5&-2\alpha_5c^2&-2\alpha_5c\\
-2\alpha_5c^{-2}&2\alpha_5&12\alpha_3+2\alpha_5c^{-2}&-2\alpha_5&-2\alpha_5c^{-1}\\
2\alpha_5&-2\alpha_5c^2&-2\alpha_5&12\alpha_4+2\alpha_5c^2&2\alpha_5c\\
2\alpha_5c&-2\alpha_5c&-2\alpha_5c^{-1}&2\alpha_5c&-4\alpha_5\\
\end{array}
\right).\end{equation}

The matrix $M$ is the \textit{Gram matrix} of $h_p(\x,\y)$ with
respect to the basis $\s=(s_1,s_2,s_3,s_4,s_5)^T$. This means that
$h_p(\x,\y)=\s^TM\s$. To show that $p$ is sos-convex, it suffices
to show that $M$ is a positive semidefinite matrix for all
$\alpha_i$ allowed in $S_{a,b}$.

We note that with $\alpha_i > 0$ for $i\in\{1,\ldots,4\}$ and
$\alpha_5=0,$ the matrix $M$ is diagonal with four positive entries
and therefore it is positive semidefinite, with a single zero
eigenvalue. Now we look at what happens if $\alpha_5$ is allowed
to be negative.

A direct computation shows that $$\det
M=-\frac{20736\alpha_1\alpha_2\alpha_3\alpha_4\alpha_5\left(4a^2b^2+\alpha_5\left(\frac{b^4}{\alpha_1}+\frac{a^4}{\alpha_2}+\frac{b^4}{\alpha_3}+\frac{a^4}{\alpha_4}\right)\right)}{a^2b^2}.$$

Therefore, we see that for
$-\frac{4a^2b^2}{\frac{b^4}{\alpha_1}+\frac{a^4}{\alpha_2}+\frac{b^4}{\alpha_3}+\frac{a^4}{\alpha_4}}\leq
\alpha_5 \leq 0,$ the determinant of $M$ is nonnegative and
strictly positive for $\alpha_5$ strictly between these bounds. It
follows that $M$ is positive semidefinite with
$\alpha_5=-\frac{4a^2b^2}{\frac{b^4}{\alpha_1}+\frac{a^4}{\alpha_2}+\frac{b^4}{\alpha_3}+\frac{a^4}{\alpha_4}}$
since it started with $4$ positive eigenvalues at $\alpha_5=0$
and the product of the eigenvalues is positive as $\alpha_5$ moves
from $0$ to
$-\frac{4a^2b^2}{\frac{b^4}{\alpha_1}+\frac{a^4}{\alpha_2}+\frac{b^4}{\alpha_3}+\frac{a^4}{\alpha_4}}$.

\end{proof}
%Now we show that if $p \in S_{a,b}$ and any of the defining
%inequalities for $\alpha_i$ is strict then $h_p(\x,\y)$ has an
%additional zero different from $\u_1$ and $\u_2$.

Now we show that if $p \in S_{a,b}$ and $\alpha_5$ is at its lower
bound, then $h_p(\x,\y)$ has an additional zero different from
$\u_1$ and $\u_2$.

\begin{lemma}\label{LEMMA Nontrivial Zero}
Let $p \in S_{a,b}$ with $\alpha_1, \ldots, \alpha_4 > 0$ and
$\alpha_5=-\frac{4a^2b^2}{\frac{b^4}{\alpha_1}+\frac{a^4}{\alpha_2}+\frac{b^4}{\alpha_3}+\frac{a^4}{\alpha_4}}$.
Then $h_p(\x,\y)$ has an additional zero with $x_1,x_2,y_1,y_2
\neq 0$.
\end{lemma}

\begin{proof}
When
$\alpha_5=-\frac{4a^2b^2}{\frac{b^4}{\alpha_1}+\frac{a^4}{\alpha_2}+\frac{b^4}{\alpha_3}+\frac{a^4}{\alpha_4}}$
the matrix $M$ from \eqref{EQN Matrix M} is singular and the
nullspace of $M$ is spanned by the vector
$$\v=\left(\frac{2ab^3}{\alpha_1},-\frac{2a^3b}{\alpha_2},-\frac{2ab^3}{\alpha_3},\frac{2a^3b}{\alpha_4},\frac{b^4}{\alpha_1}+\frac{a^4}{\alpha_2}+\frac{b^4}{\alpha_3}+\frac{a^4}{\alpha_4}\right)^T.$$

We would like to show that $h_p(\x,\y)=\s^TM\s$ has a nontrivial
zero. Therefore we need to show that the system of equations:
$$s_1=v_1, \, s_2=v_2, \, s_3=v_3, s_4=v_4, \, s_5=v_5,$$
where $s_i$ are defined in \eqref{EQN Definition s_i}, has a real
solution in $\x$ and $\y$. Note that since $a,b,\alpha_i \neq 0$,
it follows that this solution must have $x_1,x_2,y_1,y_2 \neq 0$.

Using the fact that
$y_3(ax_2-bx_1)=s_5+ab\left(\frac{s_3-s_1}{a^2}-\frac{s_4-s_2}{b^2}\right)$
and equations $s_1=v_1, \, s_2=v_2, \, s_5=v_5,$ we can express
$y_1,y_2,y_3$ and $x_3$ in terms of $x_1$ and $x_2$ and substitute
these into the equation $s_3=v_3$. After simplification, we get a
quadratic equation in $x_1$ and $x_2$:
\begin{align*}
-2a^3b\alpha_1\alpha_3\alpha_4(a^4\alpha_1\alpha_2\alpha_3+a^4\alpha_1\alpha_3\alpha_4-b^4\alpha_1\alpha_2\alpha_4-b^4\alpha_2\alpha_3\alpha_4)x_1^2\\
+(2 \alpha_4^2b^8\alpha_3 \alpha_2^2\alpha_1+2 \alpha_1^2 a^8 \alpha_4 \alpha_3^2 \alpha_2 -6 a^4 \alpha_2 b^4 \alpha_3^2 \alpha_4^2 \alpha_1 -2 \alpha_4 b^4 \alpha_3^2 \alpha_2^2 \alpha_1 a^4 -2 \alpha_1^2 a^4 \alpha_4^2 \alpha_3 \alpha_2 b^4\\
 +2 \alpha_1^2 a^4 \alpha_2^2 \alpha_3 b^4 \alpha_4 +\alpha_4^2 b^8 \alpha_3^2 \alpha_2^2 +\alpha_1^2 a^8 \alpha_4^2 \alpha_3^2 +\alpha_1^2 a^8 \alpha_2^2 \alpha_3^2 +\alpha_1^2 \alpha_2^2 b^8 \alpha_4^2 )x_1x_2\\
+2ab^3\alpha_2\alpha_3\alpha_4(a^4\alpha_1\alpha_2\alpha_3+a^4\alpha_1\alpha_3\alpha_4-b^4\alpha_1\alpha_2\alpha_4-b^4\alpha_2\alpha_3\alpha_4)x_2^2=0.
\end{align*}
We note that the coefficients of $x_1^2$ is almost the negative of
the coefficient of $x_2^2$. With positive $\alpha_i$ it follows
that the discriminant of this equation is positive and therefore
it always has a real solution.
\end{proof}

We now finish the proof of Theorem \ref{THM Cone Desc}.

\begin{proof}[Proof of the first claim in Theorem \ref{THM Cone Desc}]
First we claim that $T_{a,b}$ is contained in the linear span
of \aaa{$q_1,\ldots,q_5$}. Observe that these polynomials are linearly independent ternary quartics, and furthermore they are all contained in the vector space $L_{a,b}$. By Lemma \ref{lem:dimcount},
we know that $\dim L_{a,b}=5$ and therefore $q_1,\ldots,q_5$ are a basis of $L_{a,b}$. The claim now follows.

% Each zero of $h_p$ imposes $5$ linear conditions on $p$,
%due to nonnegativity of $h_p$. Suppose that $h_{p}$ is nonnegative
%and $h_p(\c,\d)=0$ with $\c \neq \d$. Then by positive
%semidefiniteness of the Hessian matrix $H_p$ we know that $H_p(\c)
%\cdot \d=0$ and $H_p(\d) \cdot \c=0$. This imposes $5$ linear
%conditions on $p$. It is easy to check by direct computation that
%the linear conditions imposed by the zeroes $\u_1$ and $\u_2$ are
%linearly independent, and therefore $T_{a,b}$ lies in a
%subspace of codimension at least $10$ and has dimension $5$. Since
%we have $h_{q_i}(\u_1)=h_{q_i}(\u_2)=0$ for all $i$, the claim
%follows.

Let $p=\alpha_1q_1+\ldots+\alpha_5q_5$ and suppose that $p \in
T_{a,b}$. We now show that $\alpha_1,\ldots,\alpha_4 \geq 0$
and $\alpha_5 \leq 0$. Let $\v=([0,b,1],[a,0,1])^T$. Then $h_{q_i}(\v)=0$ for $i\in\{1,\ldots,4\}$ while $h_{q_5}(\v)=-4a^2b^2$. Since $h_p$ is a nonnegative
biquadratic form it follows that $\alpha_5 \leq 0$.

Similarly, we can find a common zero $\v$ for any four $h_{q_i}$
with the $5$-th $h_{q_j}$ not equal to zero at $\v$, which
determines the sign of $\alpha_j$. For example, let
$\v=([0,(2+\sqrt{3})b,1)],[a,0,1])^T$. Then $h_{q_i}(\v)=0$ for $i
\neq 4$ and $h_{q_4}(\v)=(48+24\sqrt{3})b^4$ and therefore
$\alpha_4 \geq 0$.

Finally, we claim that if any $\alpha_i=0$ for $i\in\{1,\ldots,4\}$, then $\alpha_5=0$. For example, suppose that $\alpha_1=0$. We
already know that $\alpha_5 \leq 0$. We need to exhibit a point
$\v$ for which $h_{q_5}(\v)>0$ and
$h_{q_2}(\v)=h_{q_3}(\v)=h_{q_4}(\v)=0$, as this will imply that
$\alpha_5 \geq 0$ and we will be done. This occurs for
$\v=([a,b,1],[1,0,1])^T$. It is easy to construct similar examples
for $i=2,3,4$ as well.

Therefore we may restrict ourselves to the case of strictly
positive $\alpha_1,\ldots, \alpha_4$. Let
$\bar{p}=\alpha_1q_1+\ldots+\alpha_4q_4$. We know that $\bar{p}$
is convex and since $\alpha_5 \leq 0$, we just need to know the
lowest value of $\alpha_5$ so that $p=\bar{p}+\alpha_5q_5$ is
convex.

We note that
$$h_{\bar{p}}(\x,\y)=12(\alpha_1x_1^2y_1^2+\alpha_2x_2^2y_2^2+\alpha_3(x_1-ax_3)^2(y_1-ay_3)^2+\alpha_4(x_2-bx_3)^2(y_2-by_3)^2).$$

Therefore $h_{\bar{p}}(\x,\y)$ has zeroes only at the points where
either $x_1=0$ or $y_1=0$. Now, for
$\alpha_5=-\frac{4a^2b^2}{\frac{b^4}{\alpha_1}+\frac{a^4}{\alpha_2}+\frac{b^4}{\alpha_3}+\frac{a^4}{\alpha_4}},$
we know that $p=\bar{p}+\alpha_5q_5$ is sos-convex by Lemma
\ref{lem:Sa,b.sos-convex}. However, by Lemma \ref{LEMMA Nontrivial
Zero}, the \aaa{biquadratic form} $h_{p}(\x,\y)$ has a zero at a point where
$h_{\bar{p}}(\x,\y)$ is strictly positive. It follows that
$-\frac{4a^2b^2}{\frac{b^4}{\alpha_1}+\frac{a^4}{\alpha_2}+\frac{b^4}{\alpha_3}+\frac{a^4}{\alpha_4}}$
is the smallest $\alpha_5$ can be for $p=\bar{p}+\alpha_5q_5$ to be convex.
\end{proof}

\aaa{We end by a few remarks around Theorem~\ref{thm:convex=sos.covex.ternary.quartics}. In~\cite{CLRrealzeros}, Choi, Lam, and Reznick show that a nonnegative quartic form in 4 variables that has more than 11 nontrivial zeroes, or a nonnegative sextic form in 3 variables that has more than 10 nontrivial zeroes, must be a sum of squares. \aaa{The bound for the first claim was improved to 10 in~\cite{MR2999301}.} The next corollary is of the same spirit, but in relation to convex forms.
\begin{corollary}
    Let $p$ be a convex form in 4 variables of degree 4, or a convex form in 3 variables (of any degree). If $p$ vanishes at a nonzero point, then $p$ is sos-convex (and sos).
\end{corollary}

\begin{proof}
 By~\cite[Prop. 4.1]{Blenders_Reznick}, a convex form that vanishes at a nonzero point can be written, after a nonsingular linear change of coordinates, as a convex form in one fewer variable. The claim for convex forms in 4 variables of degree 4 then follows from our Theorem~\ref{thm:convex=sos.covex.ternary.quartics}. Similarly, the claim for convex forms in 3 variables follows from the fact that bivariate convex forms are sos-convex~\cite[Thm. 5.4]{AAA_PP_table_sos-convexity}. 
 
%  $p$ must be sos-convex.

% Suppose $p$ is a convex form in 4 variables of degree 4 that vanishes at a nonzero point. By~\cite[Prop. 4.1]{Blenders_Reznick}, after a nonsinglular linear change of coordinates, $p$ can be written as a form in 3 variables of degree 4. By Theorem~\ref{thm:convex=sos.covex.ternary.quartics}, $p$ must be sos-convex.
\end{proof} 
}

\aaa{Together with the results in~\cite{AAA_PP_table_sos-convexity}, Theorem~\ref{thm:convex=sos.covex.ternary.quartics} characterizes all dimensions and degrees for which one can have convex forms that are not sos-convex. As mentioned earlier, these turns out to be the same dimensions and degrees for which there are nonnegative forms that are not sums of squares~\cite{Hilbert_1888}, though for very different reasons. An interesting remaining question is to characterize dimensions and degrees for which one can have convex forms that are not sums of squares. Existence of such forms was shown in~\cite{Blekherman_convex_not_sos} (see also \cite[Chapter 4]{MR3075433}) by Blekherman when the degree is 4 or larger and the dimension is large enough. El~Khadir has shown that such a quartic form does not exist in dimension 4~\cite{Bachir_SIAGA}, and Saunderson has constructed an explicit example in dimension 272~\cite{saunderson2023convex}.}

% \aaa{[Do we know (or care) how many squares we will need in an sos
% decomposition of Hessian forms of ternary quartics? (Just because
% Hilbert's result also tells you the number of squares.)]}

% \vspace{5mm}
%  \texttt{AAA: Here's an immediate corollary of
% Theorem~\ref{thm:convex=sos.covex.ternary.quartics} and a
% different \\ characterization of sos-convexity given
% in~\cite{AAA_PP_CDC10_algeb_convex} that we may or may not want to
% \\include in the paper depending on whether you find it interesting. \\ AAA 2012: I no longer find this interesting.}

% \begin{corollary}
% Let $f$ be a senary quartic form in the variables
% $x\mathrel{\mathop:}=(x_1,x_2,x_3)$ and
% $y\mathrel{\mathop:}=(y_1,y_2,y_3)$ that can be written as
% $$\textstyle{\frac{1}{2}p(x)+\frac{1}{2}p(y)-p\left(\frac{x+y}{2}\right)}$$ for some ternary quartic form
% $p$. Then $f$ is psd if and only if it is sos.
% \end{corollary}

\bibliographystyle{abbrv}
\bibliography{pablo_amirali}

\end{document}